\documentclass{article}%
\usepackage{amsmath}
\usepackage{amsfonts}
\usepackage{amssymb}
\usepackage{graphicx}%
\providecommand{\U}[1]{\protect\rule{.1in}{.1in}}
\newtheorem{theorem}{Theorem}

\newtheorem{corollary}[theorem]{Corollary}

\newtheorem{definition}[theorem]{Definition}

\newtheorem{lemma}[theorem]{Lemma}

\newtheorem{proposition}[theorem]{Proposition}
\newtheorem{remark}[theorem]{Remark}

\newenvironment{proof}[1][Proof]{\noindent\textbf{#1.} }{\ \rule{0.5em}{0.5em}}
\begin{document}

\title{$2$-Graded Identities for the Tensor Square of the Grassmann Algebra}

\title{On the Graded Identities for Elementary Gradings in Matrix Algebras over Infinite Fields}

\author{Diogo Diniz Pereira da Silva e Silva \footnote{Supported by CNPq}}
\date{{\small \textit{Unidade Acadêmica de Matemática e Estatística\\
Universidade Federal de Campina Grande\\
Cx. P. 10.044, 58429-970, Campina Grande, PB, Brazil}}\\
E-mail: diogo@dme.ufcg.edu.br}
\maketitle



\maketitle

\begin{abstract}
 We consider the algebra $E\otimes E$ over an infinite field equipped with a $\mathbb{Z}_2$-grading where the canonical basis is homogeneous and prove that in various cases the graded identites are just the ordinary ones. If the grading is a non-canonical grading obtained as a quotient grading of the natural $\mathbb{Z}_2\times\mathbb{Z}_2$-grading we exhibit a basis for the graded identities.
\end{abstract}

\section{Introduction}

In \cite{KrakowskiRegev} Krakovski and Regev proved that, for fields of characteristic zero, the polynomial identities of the Grassmann algebra are consequences of the identity $[x_1,x_2,x_3]$. In \cite{GiambrunoKoshlukov} it was proved that the last result holds for infinite fields. The identities of $E\otimes E$ were considered in \cite{Popov} where A. Popov proved that the polynomials $[x_1,x_2,[x_3,x_4],x_5]$ and $[[x_1,x_2]^2,x_2]$ form a basis for $T(E\otimes E)$ if the base field is of characteristic zero. We remark that determining a basis for $E\otimes E$ over fields of positive characteristic is still an open problem. 

The problem of determining a basis for the polynomial identities of a given algebra is hard and was solved in a few cases only, and it is usefull to consider variations of the notion of polynomial identity such as graded polynomial identities. In \cite{Kemmer} Kemer develops a structural theory of T-ideals used in his solution to the Specht problem in characteristic zero, and the concept of $\mathbb{Z}_2$-graded identity was a key component. Graded identities have various applications in PI-Theory and shortly afterwards became an object of independent studies. 

In \cite{BahturinGiambrunoRiely} the authors proved that if $G$ is a finite group and $A$ is a $G$-graded algebra then $A$ satisfies a polynomial identity if and only if the component $A_0$ is a PI-algebra. Moreover graded identities provide a usefull way to prove that two algebras satisfy the same polynomial identities. In \cite{DiVincenzo} O. M. Di Vincenzo determined a basis for the $\mathbb{Z}_2$-graded identities of $M_{11}(E)$, where $E$ denotes the  Grassmann algebra of a vector space of infinite dimension over a field of characteristic zero. This results allowed him to give a new proof that over fields of characteristic zero this algebra is PI-equivalent to $E\otimes E$. This equivalence is a consequence of Kemmer's Tensor Product Theorem which was first proved by Kemmer using his structural theory of T-ideals, we mention that another proof of this theorem was given by A. Regev in \cite{Regev} using graded identities and in \cite{AzevedoFidelisKoshlukov} the methods developed by Regev were used to prove that a multilinear version of the Tensor Product Theorem holds when the characteristic of the field is $p>0$. 

Therefore describing the gradings and the corresponding graded identities of a given algebra is an important problem. Bases for the graded identities for the matrix algebra $M_n(K)$ were determined in \cite{Vasilovsky1}, \cite{Vasilovsky2} for its natural $\mathbb{Z}$ and $\mathbb{Z}_n$ gradings for fields $\mathbb{K}$ of characteristic zero. In \cite{Azevedo2}, \cite{Azevedo1} it was proved that these results hold for infinite fields. Graded identities for elementary gradings of $M_n(K)$ were studied in \cite{drenskyebahturin} for fields of characteristic zero and in \cite{Diogo} for infinite fields. The $\mathbb{Z}_2$-graded identities of $E$, for any grading where the generating vector space is graded, were described in \cite{VivianeDiVincenzo} for fields of characteristic zero and in \cite{Centrone} for infinite fields. We mention that the graded codimentions were calculated in \cite{Viviane}. The graded identities of the algebra $E\otimes E$, with its canonical grading, over infinite fields were studied in \cite{PlamenAzevedo}. 

In this paper we consider the $\mathbb{Z}_2$-graded identities for the tensor square of the Grassmann algebra $E\otimes E$ for $\mathbb{Z}_2$-gradings where each of the two $E$ components have a grading where the generating space is homogeneous. We find a basis for the identities if the grading is a non-canonical grading obtained as a quotient grading of the natural $\mathbb{Z}_2\times\mathbb{Z}_2$-grading and prove that in various cases the graded identities are just the ordinary ones.

\section{Preliminaries}

In this paper $\mathbb{K}$ denotes an infinite field of characteristic different from $2$, all algebras and vector spaces will be considered over $\mathbb{K}$. 

Let $V$ be a vector space with coutable basis $\{e_1,e_2,\dots, e_n,\dots\}$, the \textbf{Grassmann algebra} of $V$, wich we denote by $E$, is the associative algebra with basis (as a vector space) consisting of $1$ and the products $e_{i_1}e_{i_2}\dots e_{i_k}$, where $i_1<i_2<\dots<i_k$, and multiplication determined by $e_ie_j+e_je_i=0$.

A $G$-grading on an associative algebra $A$, where $G$ is a group, is a decomposition $A=\oplus_{g \in G}A_g$ such that $A_{g}A_{h}\subset A_{gh}$. We say that an element in $A_g$ is homogeneous of degree $g$. If $H$ is a normal subgroup of $G$ the quotient $G/H$-grading is defined as $A=\oplus_{\overline{g}\in G/H}A_{\overline{g}}$, where $A_{\overline{g}}=\oplus_{h \in H} A_{gh}$.

In \cite{Centrone} the author considers gradings on $E$ such that the vector space $V$ is a graded subspace, these are described in terms of a mapping \linebreak[4] $||\cdot||:\{e_1,e_2\dots,e_n,\dots\} \rightarrow \mathbb{Z}_2$. Given such a mapping we define \linebreak[4] $||e_{i_1}e_{i_2}\dots e_{i_k}||=||e_{i_1}||+\dots ||e_{i_k}||$ and this determines a grading in $E$. The gradings induced by the maps $||\cdot||_{k^{*}}$, $||\cdot||_{\infty}$ and $||\cdot||_{k}$ defined by

\vspace{0.5cm}

$||e_i||_{k^{*}}=\left\{ \begin{array}{l}
	
	1 \mbox{ for } i=1,2,\dots,k\\
	0 \mbox{ otherwise,}
\end{array}\right.$

\vspace{0.5cm}

$||e_i||_{\infty}=\left\{ \begin{array}{l}
	
	1 \mbox{ for } i \mbox{ even }\\
	0 \mbox{ otherwise,}
\end{array}\right.$

\vspace{0.5cm}

$||e_i||_{k}=\left\{ \begin{array}{l}
	
	0 \mbox{ for } i=1,2,\dots,k\\
	1 \mbox{ otherwise,}
\end{array}\right.$
\vspace{0.5cm}

\noindent were considered. We denote by $E_{k^{*}}$, $E_{\infty}$ and $E_{k}$ respectively the Grassmann algebra with each of these gradings.

Given $\mathbb{Z}_2$-graded algebras $R$ and $Q$ their tensor product $A=R\otimes Q$ has a natural $\mathbb{Z}_2$-grading $A=A_0\oplus A_1$, where 
$A_0=(R_0\otimes Q_0) \oplus (R_1\otimes Q_1)$ and \linebreak[4] $A_1=R_0\otimes Q_1 \oplus R_1\otimes Q_0$. We denote by $g(a)$ the degree of a homogeneus element $a$ in this grading. 

The algebra $A$ also has a natural $\mathbb{Z}_2\times\mathbb{Z}_2$-grading, we define $A_{(i,j)}=A_i\otimes A_j$, $i,j \in \mathbb{Z}_2$. If $A$ denotes the algebra $E\otimes E$ with its natural $\mathbb{Z}_2\times \mathbb{Z}_2$-grading and $H$ is a proper subgroup of $G=\mathbb{Z}_2\times \mathbb{Z}_2$ then $G/H\cong\mathbb{Z}_2$ and there are three possible quotient gradings: $E_0\otimes E_0$ which is the canonical $\mathbb{Z}_2$-grading and two non-canonical isomorphic gradings: 
$E_{0^{*}}\otimes E_0$ and $E_0\otimes E_{0^{*}}$. In this paper we consider the following gradings:

\begin{enumerate}
\item[(i)] $E_{0^{*}}\otimes E_0$ and $E_0\otimes E_{0^{*}}$ - these are the non-canonical quotient gradings and are studied in Section 3 where a finite basis is determined;
\item[(ii)] $E_{\infty}\otimes E_{k}$, $E_{\infty}\otimes E_{j^{*}}$, $E_{\infty}\otimes E_{\infty}$, $E_{k}\otimes E_{\infty}$ and $E_{j^{*}}\otimes E_{\infty}$ - these are the gradings with an $E_{\infty}$ component and are studied in Section 4. We prove that the graded identities are essentially the ordinary ones,
\item[(iii)] $E_{k^{*}}\otimes E_{j^{*}}$ - the graded identities are studied in Section 5, in this case the description of the graded identities is done modulo the ordinary ones.
\end{enumerate}

Let $Y$ and $Z$ denote two disjoint countable sets of noncommutative variables and $\mathbb{K}\langle X \rangle$ denote the free associative algebra freely generated by $X$, where $X=Y\cup Z$. This algebra is $2$-graded if we impose that the variables $Y$ have degree $0$ and the variables $Z$ have degree $1$. We say that a polynomial $f(x_1,\dots, x_n) \in \mathbb{K}\langle X \rangle $ is a \textbf{polynomial identity} in an associative algebra $A$ if $f(a_1,\dots, a_n)=0$ for any $a_1,\dots, a_n$ in $A$. Moreover if $A=A_0\oplus A_1$ is a $\mathbb{Z}_2$-graded algebra a polynomial $f(y_1,\dots, y_m,z_1,\dots, z_n) \in \mathbb{K}\langle X \rangle$ is a \textbf{$\mathbb{Z}_2$-graded polynomial identity} (or $2$-graded polynomial identity) if $f(a_1,a_2,\dots,a_m,b_1,\dots,b_n)=0$ for any $a_1,a_2,\dots, a_n \in A_0$ and any $b_1, \dots, b_n \in A_1$. We denote by $T(A)$ the set of all polynomial identities of the algebra $A$ and by $T_2(A)$ the set of all 
$2$-graded polynomial identities of the graded algebra $A=A_0\otimes A_1$. It is well known that the set $T(A)$ (resp. $T_2(A)$) is an ideal of $\mathbb{K}\langle X \rangle$ stable under all endomorphisms (resp. graded endomorphisms) of $\mathbb{K}\langle X \rangle$, we call such ideals $T$-ideals (resp. $T_2$-ideals).

Let $[x_1,x_2]=x_1x_2-x_2x_1$ be the commutator of $x_1$ and $x_2$, we define inductively the higer commutators \[[x_1,x_2,\dots, x_n]=[[x_1,\dots, x_{n-1}],x_n].\] We denote by $\circ$ the Jordan product, i. e., given $u,v$ in an associative algebra we have $u\circ v = uv+vu$. In this paper $B(Y,Z)$ is the subalgebra  of $\langle X \rangle$ generated by $Z$ and by the non-trivial commutators. The elements of $B(Y,Z)$ are called \textbf{$Y$-proper polynomials}. It is well known that, since our field is infinite, all the 2-graded identities of a 2-graded algebra $A$ follow from the $Y$-proper multihomogeneous ones.

\subsection{The algebra $E\otimes E$}

We denote by $A$ the algebra $E\otimes E$. Let $C=\{e_i\otimes 1, 1\otimes e_i | i=1,2,\dots, j=1,2,\dots\}$ we totally order $C$ by imposing that $e_i\otimes 1 < 1\otimes e_j$ for every $i,j$, moreover if $m<n$ then $e_m\otimes 1 < e_n \otimes 1$ and $1\otimes e_m < 1\otimes e_n$. Then given any $b \in \textbf{\textsl{B}}$, where $\textbf{\textsl{B}}$ is the canonical basis of $A$, there exists a unique sequence $c_1<c_2< \dots< c_k$ of elements in $C$ such that $b=c_1\dots c_k$.

\begin{definition}
Let $b=c_1\dots c_k$, where $b \in \textbf{\textsl{B}}$ and $c_i \in C$. We say that the set $\{c_1,\dots,c_k\}$ is the support of $b$ and denote it by $supp(b)$.
\end{definition}

If $b_1, b_2 \in \textbf{\textsl{B}}$ then $b_1b_2\neq 0$ if and only if $b_1$ and $b_2$ have disjoint supports, moreover in this case $b_1b_2 \in \textbf{\textsl{B}}$ or $-b_1b_2 \in\textbf{\textsl{B}}$, i. e., $b_1b_2=\pm b$, where $b\in \textbf{\textsl{B}}$. Clearly the sing is determined by the degree of the elements in $supp(b_1)$ and in $supp(b_2)$ in the cannonical $\mathbb{Z}_2\times \mathbb{Z}_2$-grading of $E\otimes E$. We generalize this in the following remark.

\begin{remark}\label{r1}
Given a monomial $m(x_1,\dots,x_n) \in \mathbb{K}\langle X \rangle$ and $b_1,\dots, b_n \in \textbf{\textsl{B}}$ it follows that $m(b_1,\dots,b_n)\neq 0$ if and only if the elements in $b_1,\dots, b_n$ have disjoint supports. Moreover if $m(b_1,\dots,b_n)\neq 0$ then $m(b_1,\dots,b_n)=\pm b$, where $b \in \textbf{\textsl{B}}$, and the sign is determined by the monomial and the degrees of the elements in the sets $supp(b_1),\dots,supp(b_n)$ with respect to the canonical $\mathbb{Z}_2\times \mathbb{Z}_2$-grading of $E\otimes E$.
\end{remark}

\begin{proposition}\label{prop1}
Let $f(x_1,\dots, x_n)$ be a multihomogeneous polynomial, and let $a_l=\sum_{k=1}^{v_l} \alpha_k^l b_k^l$, $l=1,2,\dots, n$, where $b_k^l \in \textbf{\textsl{B}}$, 
$\alpha_k^l \in K$, and the elements $b_k^l$ have pairwise disjoint supports. If $c_k^l \in \textbf{\textsl{B}}$ are elements in the canonical basis of $A$ with $g(b_k^l)=g(c_k^l)$ and $a_l^{\prime}=\sum_{k=1}^{n_l} \alpha_k^l c_k^l$ then $f(a_1,\dots,a_n)=0$ implies $f(a_1^{\prime},\dots, a_n^{\prime})=0$.
\end{proposition}

\begin{proof}
It follows from the equalities $g(b_k^l)=g(c_k^l)$ that we may write \[f(a_1,\dots,a_n)=\sum_{i=1}^{\lambda}P_i(\alpha_1^{1},\dots, \alpha_n^{v_n})m_i(b_1^{1},\dots,b_n^{v_n}),\] and 
\[f(a_1^{\prime},\dots,a_n^{\prime})=\sum_{i=1}^{\lambda}P_i(\alpha_1^{1},\dots, \alpha_n^{v_n})m_i(c_1^{1},\dots,c_n^{v_n}),\] where $m_1,\dots, m_{\lambda}$ are monomials in 
$\mathbb{K}\langle X \rangle$ and $P_1,\dots,P_{\lambda}$ are polynomials in commuting variables. Since the elements $b_k^l$ have pairwise disjoint supports it follows from the previous remark that $\pm m_1(b_1^{1},\dots,b_n^{v_n}), \dots \pm m_{\lambda}(b_1^{1},\dots,b_n^{v_n})$ are distinct elements of $\in \textbf{\textsl{B}}$, therefore $P_i(\alpha_1^{1},\dots, \alpha_n^{v_n})=0$, $i=1,\dots, \lambda$, hence $f(a_1^{\prime},\dots, a_n^{\prime})=0$.
\end{proof}

\section{The non-canonical quotient gradings}

We denote by $A=A_{(0,0)}\oplus A_{(1,0)}\oplus A_{(0,1)} \oplus A_{(1,1)}$ the canonical $\mathbb{Z}_2\times \mathbb{Z}_2$-grading of the algebra $E\otimes E$. There are three possible quotient gradings: 

\begin{enumerate}

\item[(1)] $A_0=A_{(0,0)}\oplus A_{(1,1)}$ and $A_1=A_{(1,0)}\oplus A_{(0,1)}$;
\item[(2)] $A_0=A_{(0,0)}\oplus A_{(1,0)}$ and $A_1=A_{(0,1)}\oplus A_{(1,1)}$;
\item[(3)] $A_0=A_{(0,0)}\oplus A_{(0,1)}$ and $A_1=A_{(1,0)}\oplus A_{(1,1)}$.

\end{enumerate}

The grading in $(1)$ is the usual $\mathbb{Z}_2$-grading of $E\otimes E$ and the graded polynomial identities in this case were described in \cite{PlamenAzevedo}. In this section we find a basis for the graded identities of $A$ with the quotient grading in $(2)$. Note that the results in this section also hold for the grading in $(3)$. 

Denote by $I$ the ideal generated by the polynomials

\begin{equation}
[y_1,y_2,x_3],
\label{1}
\end{equation}

where $x_3$ is a a variable of even or odd degree, i. e., $x_3=y_3$ or $x_3=z_3$. 

\begin{equation}
[y_1,z_2,y_3]
\label{2}
\end{equation}

\begin{equation}
[y_1,z_2]\circ z_3
\label{3}
\end{equation}

\begin{equation}
[z_1\circ z_2, z_3]
\label{4}
\end{equation}

\begin{equation}
(z_1\circ z_2)(z_3\circ z_4) - (z_1\circ z_3)(z_2\circ z_4)
\label{5}
\end{equation}

\begin{equation}
[x_1,y_2][y_3,x_4]+[x_1,y_3][y_2,x_4]
\label{6}
\end{equation}

\begin{equation}
[y_1,z_2](z_3\circ z_4)-[y_1,z_3](z_2\circ z_4)
\label{7}
\end{equation}

\begin{lemma}\label{lemma6}
The graded identities from (\ref{1}) to (\ref{7}) hold for the algebra $E_{0^{*}}\otimes E_0$. In other words $I \subset T_2(E_{0^{*}}\otimes E_0)$.
\end{lemma}

\begin{proof}
The proof consists of straitghtforward (and easy) computations, so we omit it.
\end{proof}

\begin{lemma}\label{lemma8}
The polynomial $[z_1 \circ z_2,y_3]$ lies in $I$.
\end{lemma}

\begin{proof}
We have \[[z_1 \circ z_2,y_3]=z_1\circ [z_2,y_3]+z_2\circ [z_1,y_3],\] and the result follows from (\ref{3}) since each summand in the right side of the above equality lies in $I$.
\end{proof}

We denote by $R$ the algebra $\mathbb{K}\langle Y \cup Z \rangle / I$, we denote by $y_i$ (resp. $z_i$) the image of $y_i \in Y$ (resp. $z_i \in Z$) in the quotient $R$.

\begin{lemma}
In the relatively free algebra $R$ every polynomial $f(z_1,\dots, z_n)$ is a linear combination of the polynomials 
\begin{equation}
(z_{i_1}\circ z_{i_2})\dots(z_{i_{2k-1}}\circ z_{i_{2k}})z_{j_1}\dots z_{j_l},
\label{z}
\end{equation}
where $i_1\leq i_2 \leq \dots \leq i_{2k}$ and $j_1\leq j_2 \leq \dots \leq j_l$.
\end{lemma}

\begin{proof}
Let $W$ be the subspace of $R$ generated by the polynomials (\ref{z}) and $U$ be the subspace of $R$ of the polynomials $f(z_1,\dots, z_n)$. Clearly $W\subset U$ we will prove the reverse inclusion by contradiction. We order the monomials $z_{j_1}\dots z_{j_l}$ by the lexicographic order, let $m=z_{j_1}\dots z_{j_l}$ be the least monomial that does not lie in 
$W$, in this case $j_1\leq \dots \leq j_{a}$, and $j_a>j_{a+1}$ where $1\leq a <l$, therefore \[z_{j_1}\dots z_{j_l}=z_{j_1}\dots z_{j_{a-1}}(z_{j_{a+1}}\circ z_{j_a}-z_{j_{a+1}}z_{j_a})\dots z_{j_l}.\] It follows from (\ref{4}) that \[m=(z_{j_{a+1}}\circ z_{j_a})(z_{j_1}\dots z_{j_{a-1}}\dots z_{j_l})-z_{j_1}\dots z_{j_{a-1}}z_{j_{a+1}}z_{j_a}\dots z_{j_l},\] since $z_{j_1}\dots z_{j_{a-1}}\dots z_{j_l}<m$ and $z_{j_1}\dots z_{j_{a-1}}z_{j_{a+1}}z_{j_a}\dots z_{j_l}<m$ it follows that $z_{j_1}\dots z_{j_{a-1}}\dots z_{j_l}$ and $z_{j_1}\dots z_{j_{a-1}}z_{j_{a+1}}z_{j_a}\dots z_{j_l}$ lie in $W$, but using the identity (\ref{5}) we conclude that $(z_{j_{a+1}}\circ z_{j_a})(z_{j_1}\dots z_{j_{a-1}}\dots z_{j_l})$ lies in $W$, a contradiction.
\end{proof}

We consider the polynomials,

\begin{equation*}
h_y(y_{i_1},\dots,y_{i_{2m}})=[y_{i_1},y_{i_2}]\dots[y_{i_{2m-1}},y_{i_{2m}}],
\end{equation*}

\begin{equation*}
h_{y,z}(y_{i_{2m+1}},\dots,y_{i_{2m+n}},z_{j_1},\dots,z_{j_n})=[y_{i_{2m+1}},z_{j_1}]\dots[y_{i_{2m+n}},z_{j_n}],
\end{equation*}
and
\begin{equation*}
h_z(z_{j_{n+1}},\dots,z_{j_{n+2p}})=(z_{j_{n+1}}\circ z_{j_{n+2}})\dots (z_{j_{n+2p-1}}\circ z_{j_{n+2p}}).
\end{equation*}

Define 
\begin{equation}
h(y_{i_1},\dots,y_{i_{2m+n}},z_{j_1},\dots, z_{j_{n+2p}})=h_y h_{y,z} h_z.
\label{h}
\end{equation}
Let $K=\{k_1,\dots,k_q\}$ be a set of $q$ natural numbers $k_1<\dots<k_q$ we denote by $g_K$ the product of the monomial $z_{k_1}\dots z_{k_q}$ by $h(y_{i_1},\dots,y_{i_{2m+n}},z_{j_1},\dots, z_{j_{n+2p}})$.

\begin{remark}
Clearly using identities (\ref{5}), (\ref{6}) and (\ref{7}) we may assume, multiplying by $-1$ if necessary, that in the polynomials $g_K$ above we have $i_1< i_2< \dots < i_{2m+n}$, $j_1\leq \dots\leq j_{n+2p}$, $k_1< \dots <k_{q}$. 
\end{remark}

\begin{lemma}\label{R}
In the relatively free algebra $R$ every $Y$-proper polynomial is a linear combination of the polynomials 
\[g_K=h(y_{i_1},\dots,y_{i_{2m+n}},z_{j_1},\dots, z_{j_{n+2p}})z_{k_1}\dots z_{k_q},\]
where $i_1< i_2< \dots < i_{2m+n}$, $j_1\leq \dots\leq j_{n+2p}$, $k_1< \dots <k_{q}$.

\end{lemma}

\begin{proof}
Let $U$ be the subspace of $R$ generated by the polynomials $g_K$. From the identity (\ref{1}) we conclude that comutators $[y_a,y_b]$ lie in the center of $R$. It follows from 
(\ref{2}) and (\ref{3}) that the commutators $[y_a, z_b]$ commute with the even indeterminates and anticomute with the odd indeterminates. It is clear, using the identity (\ref{4}) and Lemma \ref{lemma8}, that the elements $z_a\circ z_b$ also lie in the center of $R$. 

Therefore the product of two polynomials of the form $g_K$ is a polynomial of the same type with the ordering of the indices following from identities (\ref{5}), (\ref{7}) and Lemma \ref{z}, thus $U$ is a subalgebra of $R$. Moreover if $f$ is a polynomial of the form $g_K$ then $[f,x_i]$ lies in $U$ for any indeterminate $x_i$, hence every commutator $[y_1,x_2,\dots, x_n]$ lies in $U$.  
\end{proof}

Recall that $A_0=A_{(0,0)} \oplus A_{(1,0)}$ and $A_1=A_{(0,1)} \oplus A_{(1,1)}$ and let $\varphi:\mathbb{K}\langle Y\cup Z \rangle \rightarrow A$ be a graded homomorphism. Given $y \in Y$ we denote by $\varphi(y)_0$ and $\varphi(y)_1$ the projections of $\varphi(y)$ in $A_{(0,0)}$ and $A_{(1,0)}$ respectively. Analogously one defines $\varphi(z)_0$ and $\varphi(z)_1$ for $z \in Z$.

\begin{lemma}\label{lemma10}
Let $\varphi:\mathbb{K}\langle Y\cup Z \rangle \rightarrow E_{0^{*}}\otimes E_0$ be a graded homomorphism and let $h(y_{i_1},\dots,y_{i_{2m+n}},z_{j_1},\dots, z_{j_{n+2p}})$ be the polynomial defined in (\ref{h}). We have \[\varphi(h)=\pm2^{m+n}\left(\prod_{k=1}^{2m+n} \varphi(y_{i_k})_1\right)\cdot \left(\prod_{l=1}^{n+2p}\varphi(z_{j_l})_1\right).\]
\end{lemma}

\begin{proof}
Since $\varphi(y_{i_k})_0$ lies in the center of $A$ and the elements $\varphi(y_{i_k})_1$ anticommute it follows that 
\begin{equation}\label{p1}
\varphi([y_{i_1},y_{i_2}]\dots[y_{i_{2m-1}},y_{i_{2m}}])=2^{m}\prod_{k=1}^{2m}\varphi(y_{i_k})_1.
\end{equation}
The elements $\varphi(y_{i_k})_1$ commute with $\varphi(z_{j_l})_0$ and anticommute with $\varphi(z_{j_l})_1$ hence 
\begin{equation}\label{p2}
\varphi([y_{i_{2m+1}},z_{j_1}]\dots[y_{i_{2m+n}},z_{j_n}])=\pm2^n \left(\prod_{k=2m+1}^{2m+n}\varphi(y_{i_k})_1\right)\cdot \left(\prod_{l=1}^{n+2p}\varphi(z_{j_l})_1\right).
\end{equation}

Moreover the elements $\varphi(z_{j_k})_0$ and $\varphi(z_{j_l})$ anticommute, therefore
\begin{equation}\label{p3}
\varphi((z_{j_{n+1}}\circ z_{j_{n+2}})\dots (z_{j_{n+2p-1}}\circ z_{j_{n+2p}}))= \left(\prod_{k=n+1}^{n+2p}\varphi(z_{j_k})_1\right).
\end{equation}

Finally multipying (\ref{p1}), (\ref{p2}) and (\ref{p3}) yelds the result.
\end{proof}

\begin{lemma}\label{lemma11}
Let $h$ be the polynomial in the previous lemma and $\varphi:\mathbb{K}\langle Y\cup Z \rangle \rightarrow A$ be a graded homomorphism such that 
$\varphi(z_{j_k})_1=\sum_{l=1}^{n_k} \alpha_k b_l^k$, where $b_k$ is in the cannonical basis $\textbf{\textsl{B}}$ of $E_{0^{*}}\otimes E_0$ and $\alpha_k \in K$. If $deg_{z_{j_k}}>n_k$ for some $k$ then $\varphi(h)=0$, moreover if $deg_{z_{j_k}}=n_k$ for every $k$ then \[\varphi(h)=\pm2^{m+n}\left(\prod_{k=1}^{n+2p}\alpha_k n_k!\right)\cdot\left(\prod_{k=1}^{2m+n} \varphi(y_{i_k})_1\right)\cdot \left(\prod_{k=1}^{n+2p} \left(\prod_{l=1}^{n_k}  b_l^k\right) \right).\]
\end{lemma}

\begin{proof}
We have $\prod_{l=1}^{n+2p}\varphi(z_{j_l})_1=\prod_{j_k}(\sum_{l=1}^{n_k} \alpha_k b_l^k)^{d_k}$, where $d_k$ is the degree of $z_{j_k}$ and the last product runs over all $j_k$ such that $z_{j_k}$ appears in $h$. Since the $b_l^k$ commute and $(b_l^k)^2=0$ it follows that 
$(\sum_{l=1}^{n_k} \alpha_k b_l^k)^{n_k}= (n_k!)\prod_{l=1}^{n_k} \alpha_k b_l^k$ and $(\sum_{l=1}^{n_k} \alpha_k b_l^k)^{d_k}=0$ if $d_k>n_k$, hence the result follows from the previous lemma.
\end{proof}

If $\mathbb{K}$ is a field of characteristic $p>2$ let $I_p$ denote the $T_2$-ideal generated by the identities (\ref{1})-(\ref{7}) together with the identities

\begin{equation}\label{zp}
[y_1,z_1]\dots[y_{2k-2},z_1](z_2\circ z_1)z_1^{2n}
\end{equation}
and the identities
\begin{equation}\label{p}
[y_1,z_1]\dots[y_{2k-1},z_1]z_1^{2n},
\end{equation}
where $2n+2k-1=p$, $n=0,1,\dots, \frac{p-1}{2}$.

We denote by $R_p$ the algebra $\mathbb{K}\langle Y \cup Z \rangle/I_p$.

\begin{lemma}\label{lemma13}
If $\mathbb{K}$ is a field of characteristic $p>2$ then $I_p\subset T_2(E_{0^{*}}\otimes E_0)$.
\end{lemma}

\begin{proof}
It follows from Lemma \ref{lemma10} that for any graded homomorphism $\varphi:\mathbb{K}\langle Y\cup Z \rangle \rightarrow A$ we have \[\varphi([y_1,z_1]\dots[y_{2k-1},z_1]z_1^{2n})=\pm2^{2k-1}\left(\prod_{i=1}^{2k-1} \varphi(y_{i})_1\right)\cdot \left(\varphi(z_{j_l})_1\right)^p,\] where $2n+2k-1=p$. We have $\varphi(z_{j_l})_1=\sum_{l=1}^n \alpha_l b_l$, where $b_l \in \beta$, in this case the elements $b_l$ commute and $(b_l)^2=0$ for each $l$, hence \[(\sum_{l=1}^n \alpha_l b_l)^p=\sum_{\stackrel{l_i \neq l_j}{1\leq i<j\leq p}}b_{l_1}\dots b_{l_p}=\sum_{l_1<\dots< l_p}(p!)b_{l_1}\dots b_{l_p}=0.\]Analogously $\varphi(z_1^{p+1})=0$,  therefore $[y_1,z_1]\dots[y_k,z_1]z_1^{2n}$ and $z_1^{p+1}$ are graded identities for $E_{0^{*}}\otimes E_0$ and now the result follows from Lemma \ref{lemma6}.
\end{proof}

\begin{corollary}\label{corolp}
In the algebra $R_p$ every $Y$-proper polynomial is a linear combination of the polynomials \[g=h(y_{i_1},\dots,y_{i_{2m+n}},z_{j_1},\dots, z_{j_{n+2p}})z_{k_1}\dots z_{k_q},\]
where $i_1< i_2< \dots < i_{2m+n}$, $j_1\leq \dots\leq j_{n+2p}$, $k_1< \dots <k_{q}$ and $deg_{z_l}h<p$, $1\leq l \leq n+2p$.
\end{corollary}

\begin{proof}
For each $l$ if $deg_{z_{j_l}}h\geq p$ then $h$ is a consequence of (\ref{zp}) or (\ref{p}), hence from Lemma \ref{lemma13} it follows that $h=0$, now the result follows from Lemma \ref{R}.
\end{proof}

\begin{theorem}
If $\mathbb{K}$ is a field of characteristic zero then the $2$-graded identities of $E_{0^{*}}\otimes E_0$ follow from the identities (\ref{1}) - (\ref{7}). If $\mathbb{K}$ is an infinite field of characteristic $p>2$ then the $2$-graded identities of $E_{0^{*}}\otimes E_0$ follow from the identities (\ref{1}) - (\ref{7}) together with the identities (\ref{zp}) and (\ref{p}).
\end{theorem}

\begin{proof}
Let $\mathbb{K}$ be an infinite field of characteristic $p>2$. Let $f(y_1,\dots,y_m, z_1\dots,z_m)$ be a multihomogeneous $Y$-proper $2$-graded identity for $E_{0^{*}}\otimes E_0$, then it follows from Corollary \ref{corolp} that in $R_p$ we have $f=\sum \alpha_K g_K$, where \[g_K=h_K(y_{1},\dots,y_{m},z_{1},\dots, z_{n})z_{k_1}\dots z_{k_q},\] and 
$K=\{k_1,\dots, k_q\}$. If $\alpha_K=0$ for every subset $K$ of $\{1,\dots, m\}$ we are done. If $\alpha_K\neq 0$ for some $K \subset\{1,\dots, m\}$, we denote by $K_0$ a minimal subset with $\alpha_{K_0}\neq 0$. Given $K \subset\{1,\dots, m\}$ let $\varphi_{K}:\mathbb{K}\langle Y\cup Z \rangle \rightarrow A$ be a $2$-graded homomorphism such that $\varphi(z_{j})_1=\sum_{l=1}^{n_j} \alpha_j b_l^j$, where $n_j$ is the degree of $z_j$ in the polynomial $h_K$,  $\varphi(z_j)_0=1\otimes e_{a_j}$ and the elements $\varphi(y_i)$, $\varphi(z_j)_1$ and 
$\varphi(z_j)_0$ have disjoint supports.

Let $K_1$ be a subset of $\{1,2,\dots,n\}$ such that $\alpha_{K_1}\neq 0$, since $K_0$ is minimal we have $K_0\not \subset K_1$. Then there exists $k_a \in K_0$ such that $k_a \notin K_1$. Since $g_{K_0}$ and $g_{K_1}$ have the same multidegree we conclude that $deg z_{k_a}(h_{K_1})>deg z_{k_a}(h_{K_0})$. Therefore $\varphi_{K_0}(z_{k_a})_1=\sum_{l=1}^{n_{k_a}} \alpha_{k_a} b_l^{k_a}$ and 
$deg z_{k_a}(h_{K_1})>n_{k_a}=deg z_{k_a}(h_{K_0})$, hence it follows from Lemma \ref{lemma11} that $\varphi_{K_0}(g_{K_1})=0$. Moreover $deg z_j(h_{K_0})<p$ therefore it follows from Lemma \ref{lemma11} that $\varphi_{K_0}(g_{K_0})\neq 0$. Hence we conclude that \[0=\varphi_{K_0}(f)=\alpha_{K_0}\varphi_{K_0}(g_{K_0}),\] which is a contradiction. Hence $\alpha_K=0$ for all $K \subset\{1,\dots, m\}$ and $f=0$ in $R_p$, therefore $f \in I_p$.

Since $f$ is an arbitrary $Y$-proper multihomogeneous polynomial we conclude that $T_2(E_{0^{*}}\otimes E_0) \subset I_p$ and the reverse inclusion follows from Lemma \ref{lemma13}. The case where $\mathbb{K}$ is a field of characteristic $0$ is analogous.
\end{proof}

\section{The Graded Identities with an $E_{\infty}$ Component}

In this section consider the algebra $A$ with one of the gradings $E_{\infty}\otimes E_{k}$, $E_{\infty}\otimes E_{j^{*}}$ or $E_{\infty}\otimes E_{\infty}$. Clearly the results in this section also hold for $E_{k}\otimes E_{\infty}$ and $E_{j^{*}}\otimes E_{\infty}$.

We denote by $\alpha(a) \in \mathbb{Z}_2$ the degree of a homogeneous element $a\in A$ and by $\alpha(f)\in \mathbb{Z}_2$ the degree of a homogeneous polynomial in 
$\mathbb{K}\langle X \rangle$. We recall that $g(a)$ denotes the degree of a homogeneous element in the $\mathbb{Z}_2\times \mathbb{Z}_2$-grading of $A$. In this section prove that the graded identities of $A$ are essentially the ordinary identities of $E\otimes E$.

\begin{lemma}\label{l3}
Given $g_1,\dots, g_n \in \mathbb{Z}_2\times \mathbb{Z}_2$ and $h_1,\dots, h_n \in \mathbb{Z}_2$ there exists $b_1,\dots, b_n \in \textbf{\textsl{B}}$ with pairwise disjoint supports such that $g(b_i)=g_i$ and $\alpha(b_i)=h_i$.
\end{lemma}

\begin{proof}
Clearly there exists $c_1,\dots, c_n \in \textbf{\textsl{B}}$ with pairwise disjoint supports such that $g(c_i)=g_i$. Let $d_l=e_{a_{2l-1}}e_{a_{2l}}\otimes 1$, where $a_1,\dots,a_{2n}$ is a sequence of pairwise different natural numbers such that the elements $a_{2l}$, $l=1,2,\dots,n$ are even and $a_{2l-1}$ is even if and only if $h_l=\alpha(c_l)$. Moreover we may choose the elements $a_i$ sufficiently large so that the elements $b_l=c_ld_l$ are different from $0$ and the elements $\pm b_l \in \textbf{\textsl{B}}$ have pairwise disjoint supports. In this case $g(b_l)=g(c_l)=g_l$ and $\alpha(b_l)=h_l$ and the lemma is proved.
\end{proof}

\begin{theorem}
If $\mathbb{K}$ is an infinite field then $f(x_1,\dots, x_n)$ is a $2$-graded identity for $A$ if and only if $f(x_1,\dots, x_n)$ is an ordinary identity for $A$.
\end{theorem}

\begin{proof}
Let $f \in T_2(A)$ be a multihomogeneous polynomial and $a_1^{\prime},\dots, a_n^{\prime}$ be arbitrary elements in $E\otimes E$. Let $a_l^{\prime}=\sum_{k=1}^{n_l} \alpha_k^l c_k^l$, where $\alpha_k^l \in K$, 
$c_k^l \in \textbf{\textsl{B}}$. It follows from Lemma \ref{l3} that there exists $b_k^l \in \textbf{\textsl{B}}$ with pairwise disjoint supports such that $g(b_k^l)=g(c_k^l)$ and the elements $a_l=\sum_{k=1}^{n_l}\alpha_{k}^l b_k^l$ are homogeneous in $A$ with $\alpha(a_l)=\alpha(x_l)$. Since $f$ is a $2$-graded identity for $A$ we have $f(a_1,\dots, a_n)=0$ and it follows from Proposition \ref{prop1} that $f(a_1^{\prime},\dots, a_n^{\prime})=0$. 
\end{proof}

\section{The Graded Identities of $E_{k^*}\otimes E_{j^*}$}

In this section we describe the graded identites of $E_{k^*}\otimes E_{j^*}$.

\begin{lemma}
The polynomial $z_1z_2\dots z_{j+k+1}$ is a graded identity for $E_{k^*}\otimes E_{j^*}$ .
\end{lemma}

\begin{proof}
Since the polynomial is multilinear we need only to consider substitutions of the indeterminates $z_i$ by elements $b_i$ of the basis $\textbf{\textsl{B}}$ with degree $1$ in the 
$\mathbb{Z}_2$-grading of $E_{k^*}\otimes E_{j^*}$. In this case $b_i$ must be divisible by an element of the set $\{e_1\otimes 1,\dots, e_k\otimes 1, 1\otimes e_1, \dots, 1 \otimes e_j\}$, hence it follows from the pigeonhole principle that the elements $b_1,\dots, b_{j+k+1}$ cannot have disjoint supports, therefore $b_1b_2\dots b_{j+k+1}=0$.
\end{proof}

\begin{lemma}\label{l4}
Let $g_1,\dots, g_n \in \mathbb{Z}_2\times \mathbb{Z}_2$ and $h_1,\dots, h_n \in \mathbb{Z}_2$. If the number of indexes $i$ in $\{1,2,\dots,n\}$ with $h_i=1$ is $\leq k+j$ then there exists $b_1,\dots, b_n \in \textbf{\textsl{B}}$ with pairwise disjoint supports such that $g(b_i)=g_i$ and $\alpha(b_i)=h_i$.
\end{lemma}

\begin{proof}
The proof is analogous to that of Lemma \ref{l3} and will be omitted.
\end{proof}

Let $I_m$ denote the $T_2$-ideal generated by the identity $z_1\dots z_{m+1}$ together with the polynomials $f(x_1,\dots, x_n) \in \mathbb{K}\langle X \rangle$ that are ordinary identities for
 $E\otimes E$.

\begin{theorem}
If $\mathbb{K}$ is an infinite field then $T_2(E_{k^*}\otimes E_{j^*})=I_{j+k}$.
\end{theorem}

\begin{proof}
It follows from Lemma \ref{l4} that $I_{k+l}\subset T_2(E_{k^*}\otimes E_{j^*})$. Let $f(x_1,\dots, x_n)$ be a multihomogeneous $2$-graded identity for $E_{k^*}\otimes E_{j^*}$, modulo the identity $z_1\dots z_{j+k+1}$ we may assume that the total degree of $f$ in the odd indeterminates is $\leq k+l$. Hence if $a_l^{\prime}=\sum_{k=1}^{n_l} \alpha_k^l c_k^l$, where $\alpha_k^l \in K$, $c_k^l \in \textbf{\textsl{B}}$ and $n_l=deg_{x_l}(f)$, it follows from Lemma \ref{l4} that there exists $b_k^l \in \textbf{\textsl{B}}$ such that $g(b_k^l)=g(c_k^l)$ and the elements $a_l=\sum_{k=1}^{n_l}\alpha_{k}^l b_k^l$ have disjoint support and are homogeneous in $A$ with $\alpha(a_l)=\alpha(x_l)$. Since $f$ is a $2$-graded identity for $A$ we have $f(a_1,\dots, a_n)=0$ and it follows from Proposition \ref{prop1} that $f(a_1^{\prime},\dots, a_n^{\prime})=0$. 
\end{proof}

\begin{flushleft}
\textbf{Acknowledgements}
\end{flushleft}
We thank CNPq for the finantial support and the referee for valuable remarks.


\begin{thebibliography}{12}

\bibitem {Azevedo2}\textrm{S. S. Azevedo},
\textit{A basis for $\mathbb{Z}$-graded identities of matrices over infinite fields}, Serdica Math. Journal \textbf{29 (2)}  (2003), 149--158.

\bibitem {Azevedo1}\textrm{S. S. Azevedo},
\textit{Graded identities for the matrix algebra of order n over an infinite field}, Comm. Algebra \textbf{30}  12 (2002), 5849--5860.

\bibitem {AzevedoFidelisKoshlukov}\textrm{S. S. Azevedo, M. Fidelis, P. Koshlukov}, \textit{Tensor produc theorems in positive characteristic}, Journal of Algebra \textbf{276} (2004), 236--245.

\bibitem{drenskyebahturin} \textrm{Y. Bahturin, V. Drensky},
\textit{Graded polynomial identities of matrices}, Linear Algebra and its Applications  (2002) 15--34.

\bibitem {BahturinGiambrunoRiely}\textrm{Yu. Bahturin, A. Giambruno, D. Riley}, \textit{Group graded algebras satisfying a polynomial identity}, Isr. J. Math. 104 (1998), 145--155.

\bibitem {Centrone}\textrm{L. Centrone}, \textit{$\mathbb{Z}_2$-graded identities of the Grassmann algebra in positive characteristic}, Linear Algebra and its Applications \textbf{435} 
(2011) 3297--3313.

\bibitem {Diogo}\textrm{D. Diniz},
\textit{On the graded identities for elementary gradings in matrix algebras over infinite fields}, Linear Algebra and its Applications \textbf{439} (2013) 1530--1537.

\bibitem{GiambrunoKoshlukov} \textrm{A. Giambruno, P. Koshlukov},
\textit{On the identities of the Grassmann algebras in characteristic $p > 0$}, Israel J. Math. 122 (1) (2001) 305--316.

\bibitem {Kemmer}\textrm{A. Kemmer}, \textit{Ideals of Identities of Associative Algebras}, Translations of Math. Monographs, \textbf{87}, AMS, Providence, RI, 1991.

\bibitem {PlamenAzevedo}\textrm{P. Koshlukov, S. S. de Azevedo}, \textit{Graded identites for T-prime algebras over fields of positive characteristic}, Israel Journal of Mathematics \textbf{128} (2002) 157--176.

\bibitem{KrakowskiRegev} \textrm{D. Krakovski, A. Regev},
\textit{The polynomial identities of the Grassmann algebra}, Trans. Amer. Math. Soc.  (1973) 429--438.

\bibitem {Popov}\textrm{A. P. Popov}, \textit{Identities for the tensor square of the Grassmann algebra}, Algebra and Logic \textbf{21} (1982), 442--471 (Russian); English Translation: Algebra and Logic \textbf{21} (1982) 296--316.

\bibitem {Regev}\textrm{A. Regev}, \textit{Tensor products of matrix algebras over the Grassmann algebra}, Journal of Algebra \textbf{133} (2) (1990) 512--526.

\bibitem {Viviane}\textrm{V. R. T. da Silva}, \textit{$\mathbb{Z}_2$-codimentions of the Grassmann algebra}, Comm. Algebra \textbf{29} (9) 3342--3359 (2011).

\bibitem {Vasilovsky1} \textrm{S. Yu. Vasilovsky},
\textit{$\mathbb{Z}$-graded polynomial identities of the full matrix algebra}, Commun. Algebra \textbf{26 (2)}, (1998) 601--612.

\bibitem {Vasilovsky2} \textrm{S. Yu. Vasilovsky},
\textit{$\mathbb{Z}_n$-graded polynomial identities of the full matrix algebra of order $n$}, Proc. Amer. Math. Soc.  \textbf{127 (12)}, (1999) 3517--3524.


\bibitem {DiVincenzo}\textrm{O. M. Di Vincenzo}, \textit{On the graded identities of $M_{1,1}(E)$}, Israel Journal of Mathematics \textbf{80} (1992) 323--335.

\bibitem {VivianeDiVincenzo}\textrm{O. M. Di Vincenzo, V. R. T. da Silva}, \textit{On $\mathbb{Z}_2$-graded polynomial identities of the Grassmann algebra}, Linear Algebra and its Applications \textbf{431} (1-2) 56--72 (2009).

\end{thebibliography}
\end{document}